\documentclass[a4paper,12pt]{amsart}

\usepackage{amsthm}
\theoremstyle{definition}
\newtheorem{theorem}{Theorem}[section]
\newtheorem{lemma}[theorem]{Lemma}

\theoremstyle{definition}
\newtheorem{definition}[theorem]{Definition}

\newtheorem{pclaim}[theorem]{Claim} 

\theoremstyle{remark}
\newtheorem{remark}[theorem]{Remark}

\theoremstyle{plain}
\newtheorem*{tight}{Tight Cut Lemma} 
\newtheorem*{rtight}{Formal Statement  of the Tight Cut Lemma} 

\numberwithin{equation}{section}

\newenvironment{romanenumerate}{
\begin{enumerate}

}
{\end{enumerate}}

\newcommand{\yield}{\triangleleft}

\newcommand{\pargpart}[2]{\mathcal{P}_{#1}(#2)}
\newcommand{\gsim}[1]{\sim_{#1}}

\newcommand{\upstar}[1]{\mathcal{U}^{*}(#1)}
\newcommand{\parup}[2]{\mathcal{U}_{#1}(#2)}
\newcommand{\parupstar}[2]{\mathcal{U}^{*}_{#1}(#2)}
\newcommand{\vup}[1]{U(#1)}
\newcommand{\vupstar}[1]{U^*(#1)}
\newcommand{\vparup}[2]{U_{#1}(#2)}
\newcommand{\vparupstar}[2]{U^*_{#1}(#2)}
\newcommand{\coup}[1]{{}^c\mathcal{U}(#1)}

\newcommand{\vcoup}[1]{{}^cU(#1)}
\newcommand{\vparcoup}[2]{{}^cU_{#1}(#2)}

\renewcommand{\Gamma}{N}

\newcommand{\complement}[1]{{#1}^c} 

\newcommand{\parNei}[2]{\Gamma_{#1}(#2)}
\newcommand{\comp}[1]{\mathcal{G}(#1)}
\newcommand{\paryield}[1]{\yield_{#1}}
\newcommand{\tower}[1]{\mathcal{T}(#1)}
\newcommand{\partower}[2]{\mathcal{T}_{#1}(#2)}

\newcommand{\poset}[1]{\mathcal{O}(#1)}

\newcommand{\pmin}[1]{{\rm min}\poset{#1}}
\newcommand{\border}[1]{\partial \poset{#1}}

\title[Tight Cut Lemma]{A Graph Theoretic Proof of the Tight Cut Lemma}
\author{Nanao Kita}
\address{National Institute of Informatics
2-1-2 Hitotsubashi, Chiyoda-ku, Tokyo, Japan 101-8430}
\email{kita@nii.ac.jp}
\date{\today}

\begin{document}
\maketitle

\begin{abstract}
In deriving their characterization of the perfect matchings polytope, 
Edmonds, Lov\'asz, and Pulleyblank introduced the so-called {\em Tight Cut Lemma} as the most challenging aspect of their work. 
The Tight Cut Lemma 
in fact claims {\em bricks} as the fundamental building blocks that constitute a graph 
in studying the matching polytope 
and can be referred to as a key result in this field.  
Even though the Tight Cut Lemma is a matching \textup{(}$1$-matching\textup{)} theoretic statement 
that consists of purely graph theoretic concepts,  
the known proofs either employ  a linear programming argument 
or are established upon results regarding a substantially wider notion than matchings. 
This paper presents a new proof of the Tight Cut Lemma, 
which attains both of the two reasonable features for the first time, namely,  
being {\em purely graph theoretic} as well as {\em purely matching theory closed}. 
Our proof uses, as the only preliminary result, the canonical decomposition recently introduced by Kita.  
By further developing this canonical decomposition, 
we acquire a new device of {\em towers} to analyze the structure of bricks, 
and thus prove the Tight Cut Lemma. 
We  believe that our new proof of the Tight Cut Lemma  provides a highly versatile example of how to handle bricks.

\end{abstract}

\section{Introduction}
Edmonds, Lov\'asz, and Pulleyblank~\cite{DBLP:journals/combinatorica/EdmondsLP82} introduced the {\em Tight Cut Lemma} as a key result in their paper characterizing the perfect matching polytope. 
They stated that proving the Tight Cut Lemma was the most difficult part. 

\begin{tight}
Any tight cut in a brick is trivial. 
\end{tight} 
A graph is a {\em brick} if  
deleting any two vertices results in a connected graph with a perfect matching. 
A cut is {\em tight} if it shares exactly one edge with any perfect matching. 
A tight cut is {\em trivial} if it is a star cut.

The Tight Cut Lemma in fact characterizes the {\em bricks} as the fundamental building blocks  that constitute a graph in the polyhedral study of matchings via the inductive operation the {\em tight cut decomposition}. 
As long as a given graph has a non-trivial tight cut, 
we can apply an operation that decomposes it into  two smaller graphs that  perfectly inherit  the matching theoretic property; this is the {\em tight cut decomposition}. 
In fact, we can view the Tight Cut Lemma as stating 
that {\em the bricks are the irreducible class of the tight cut decomposition}.   
Via the tight cut decomposition, Edmonds et al.~\cite{DBLP:journals/combinatorica/EdmondsLP82} derived the dimension of the perfect matching polytope  using  as a parameter the number of bricks that constitute a given graph. Consequently, they determine the minimal set of inequalities that defines the perfect matching polytope.    

Since Edmonds et al.~~\cite{DBLP:journals/combinatorica/EdmondsLP82}, 
the study of bricks and the consequential results on the perfect matching polytope (and lattice) have flourished; see Lov\'asz~\cite{DBLP:journals/jct/Lovasz87} and Carvalho, Lucchesi, and Murty~\cite{DBLP:journals/jct/CarvalhoLM02,DBLP:journals/jct/CarvalhoLM02a,DBLP:journals/jct/CarvalhoLM02b,DBLP:journals/jct/CarvalhoLM04,DBLP:journals/combinatorica/CarvalhoLM99}.

Edmonds et al.~\cite{DBLP:journals/combinatorica/EdmondsLP82} proves the Tight Cut Lemma via a linear programming argument, whereas the statement itself consists of purely graph theoretic notions only. 
This might be problematic as well as awkward because not knowing how to treat bricks and tight cuts combinatorially might limit our ability to investigate this field.  
Szigeti~\cite{DBLP:journals/combinatorica/Szigeti02} later gives a purely graph theoretic proof using the theory of optimal ear-decomposition proposed by Frank~\cite{Frank1993}.

In this paper, we give a new purely graph theoretic proof using the theory of {\em canonical decomposition} 
 for general graphs with perfect matchings, which was recently proposed by Kita~\cite{DBLP:conf/isaac/Kita12,DBLP:journals/corr/abs-1205-3816}. 
As the term ``canonical'' conventionally means in the mathematical context, 
canonical decompositions are a standard tool to analyze graphs in matching theory. 
Several  canonical decompositions are classically known such as the Gallai-Edmonds, 
the Kotzig-Lov\'asz, and the  Dulmage-Mendelsohn~\cite{lovasz2009matching}.   
However, none of them target the general graphs with perfect matchings but rather  more particular classes of graphs, 
until  Kita~\cite{DBLP:conf/isaac/Kita12,DBLP:journals/corr/abs-1205-3816} introduced a new canonical decomposition. 
To prove the Tight Cut Lemma, we must assume that we are given a brick, a non-star cut, and a perfect matching that shares exactly one edge, say, $e$,  with the cut,   
and then find another perfect matching that shares more than one edge with the cut. 
and then find another perfect matching that shares more than one edge with the cut. 
Deleting $e$ from the brick together with its ends results in a graph with perfect matchings. 
Hence,  analyzing the structure of this graph with Kita's canonical decomposition 
would appear to be more reasonable means of obtaining a new proof of the Tight Cut Lemma.

We further characterize our new proof as purely {\em matching \textup{(}$1$-matching\textup{)} theory closed} 
as well as purely graph theoretic.  
Our proof uses solely  the canonical decomposition by Kita for known results,  
which is obtained from scratch via the most elementary discussion regarding $1$-matchings.  
In contrast, Szigeti's proof involves explicitly or implicitly
a lot more things, some of which are not $1$-matching closed; 
because, 
the optimal ear-decomposition theory is established upon 
not only  many known results and notions in matching theory such as 
the Tutte-Berge formula and the notion of barriers, the Gallai-Edmonds decomposition, and  the theory of 
ear-decompositions of some classes of graphs,   
but also the theory of {\em $T$-join}, which is a substantially wider notion than $1$-matchings. 
As the Tight Cut Lemma and the main applications are purely $1$-matching theoretic, 
our proof has quite a reasonable nature. 
We also believe  our proof  to be significant in that it provides a highly versatile  example of how to study bricks graph-theoretically.

The remainder of this paper is organized as follows. 
Section~\ref{sec:pre} presents preliminary definitions and results: 
Section~\ref{sec:pre:def}  explains fundamental notation and definitions; 
Section~\ref{sec:pre:props} presents some elementary lemmas, 
and Section~\ref{sec:pre:cathedral}  introduces the canonical decomposition given by Kita~\cite{DBLP:conf/isaac/Kita12,DBLP:journals/corr/abs-1205-3816}. 
Section~\ref{sec:tower} introduces new results of us; 
here, we further develop a device to analyze the structure of graphs with perfect matchings, 
 which will be used in Section~\ref{sec:newproof}. 
Section~\ref{sec:newproof} gives the new proof of the Tight Cut Lemma.

\section{Preliminaries}\label{sec:pre}
\subsection{Notation and Definitions} \label{sec:pre:def}
\subsubsection{General Statements} 
For standard notations and definitions on sets and graphs, we  mostly follow Shcrijver~\cite{schrijver2003combinatorial} in this paper. In this section, we list those that are exceptional or non-standard. 
We denote the vertex set and the edge set of a  graph $G$ by $V(G)$ and  $E(G)$.  
We sometimes refer to the vertex set of a graph $G$ simply as $G$. 
As usual, we often denote a singleton $\{x\}$ simply by $x$. 

\subsubsection{Operations of Graphs} 
Let $G$ be a graph, and let $X \subseteq V(G)$. 
The subgraph of $G$ induced by $X$ is denoted by $G[X]$. 
The notation $G-X$ denotes the graph $G[V(G)\setminus X]$. 
The contraction of $G$ by $X$ is denoted by $G/X$. 
Let $\hat{G}$ be a supergraph of $G$, and let $F \subseteq E(\hat{G})$. 
The notation $G+F$ and $G-F$ denotes the graphs obtained by adding and by deleting $F$ from $G$. 
Given another subgraph $H$ of $\hat{G}$, 
the graph $G+H$ denotes the union of $G$ and $H$. 
In referring to graphs obtained by these operations, 
we often identify their items such as vertices and edges with the naturally corresponding items of old graphs. 

\subsubsection{Paths and Circuits} 
We treat paths and circuits as graphs; 
i.e., a circuit is a connected graph in which every vertex is of degree two, 
and a path is a connected graph if every vertex is of degree no more than two and it is not a circuit.   
Given a path $P$ and two vertices $x$ and $y$ in $V(P)$, 
$xPy$ denotes the connected subgraph of $P$, which is of course a path, that has the ends $x$ and $y$. 

\subsubsection{Functions on Graphs}
The set of neighbors of $X\subseteq V(G)$ in a graph $G$ is denoted by $\parNei{G}{X}$; 
namely, $\parNei{G}{X} := \{ u \in V(G)\setminus X : \exists v \in X \mbox{ s.t. } uv \in E(G) \}$. 
Given $X, Y \subseteq V(G)$,  $E_G[X,Y]$ denotes the set of edges of $G$ whose two ends are in $X$ and in $Y$. 
We denote $E_G[X, V(G)\setminus X]$ by $\delta_{G}(X)$. 
We often omit the subscripts ``$G$'' in using these notations. 

\subsubsection{Matchings} 
Given a graph , a  {\em matching} is a set of edges in which any two are disjoint. 
A matching is a {\em perfect matching} if every vertex of the graph is adjacent to one of its edges. 
A graph is {\em factorizable} if it has a perfect matching. 
An edge of a factorizable graph is {\em allowed} if it is contained in a perfect matching. 
A graph $G$ is {\em factor-critical} if it has only a single vertex or for, any $v\in V(G)$, $G-v$ is factorizable.

Given a set of edges $M$, 
a circuit $C$ is {\em $M$-alternating} if $E(C)\cap M$ is a perfect matching of $C$. 
A path $P$ with two ends $x$ and $y$ is {\em $M$-saturated} (resp. {\em $M$-exposed})  between $x$ and $y$ if $E(P)\cap M$ (resp. $E(P)\setminus M$) is a perfect matching of $P$. 
A path $P$ with ends $x$ and $y$ is {\em $M$-balanced} from $x$ to $y$ 
if  $E(P)\cap M$ is a matching of $P$ and, among the vertices in $V(P)$, only $y$ is disjoint from the edges in $E(P)\cap M$.  
We define a trivial graph, i.e., a graph with a single vertex and no edges, as an $M$-balanced path. 
In other words, if we trace an $M$-alternating circuit, or $M$-saturated, exposed, or balanced path from a vertex, then edges in $M$ and in $E(G)\setminus M$ appear alternately; 
in an $M$-saturated path,  both edges adjacent to the ends are in $M$, 
whereas in an $M$-exposed path, neither of them are, and in an $M$-balanced path, one of them is in $M$ but the other is not. 

Given a set of vertices $X$, 
an $M$-exposed path is an {\em $M$-ear} relative to $X$ 
if the ends are in $X$ while the other vertices are disjoint from $X$; 
also, a circuit $C$ is an {\em $M$-ear} relative to $X$ 
if $V(C)\cap X = \{ x \}$ holds and $C-x$ is an $M$-saturated path.  
In the first case, we say the $M$-ear is {\em proper}. 
Even in the second case, we call $x$ an end of the $M$-ear for convenience.  
An $M$-ear is {\em trivial} if it consists of only a single edge. 
We say an $M$-ear {\em traverses} a set of vertices $Y$ if it has a vertex other than the ends that is in $Y$.

\subsection{Fundamental Properties} \label{sec:pre:props}
We now present elementary lemmas that will be used in later sections. 
They are easy to confirm. 
\begin{lemma}\label{lem:delete2path} 
Let $G$ be a factorizable graph, and $M$ be a perfect matching of $G$. 
Given  two distinct vertices $u,v\in V(G)$, 
$G-u-v$ is factorizable if and only if there is an $M$-saturated path between  $u$ and $v$. 
\end{lemma} 

\begin{lemma}\label{lem:circ2matching} 
Let $G$ be a factorizable graph and $M$ be a perfect matching of $G$. 
Let $C$ be an $M$-alternating circuit of $G$. 
Then\footnote{We denote the symmetric difference of two sets $A$ and $B$ by $A\triangle B$. }, $M\triangle E(C)$ is a perfect matching of $G$, 
and therefore the edges of $C$ are all allowed.  
\end{lemma}

\subsection{Canonical Decomposition for General Factorizable Graphs} 
\label{sec:pre:cathedral} 
We now introduce the canonical decomposition given by Kita~\cite{DBLP:conf/isaac/Kita12,DBLP:journals/corr/abs-1205-3816}, 
which will be used in Sections~\ref{sec:tower} and \ref{sec:newproof} as the only preliminary result to derive the Tight Cut Lemma.  
The principal results that constitute the theory of this canonical decomposition  
are Theorems~\ref{thm:order}, \ref{thm:sim}, and \ref{thm:cor}. 
In this section, unless otherwise stated, 
$G$ denotes a factorizable graph.  
\begin{definition} 
Let $\hat{M}$ be the union of all perfect matchings of $G$.  
A {\em factor-component} of $G$ is the subgraph induced by $V(C)$,  
where $C$ is a connected component of the subgraph of $G$ determined by $\hat{M}$. 
The set of factor-components of $G$ is denoted by $\comp{G}$. 
That is to say, a factorizable graph consists of factor-components and edges joining distinct factor-components. 
A {\em separating set} of $G$ is a set of vertices that is the union of the vertex sets of some factor-components of $G$. 
Note that if $X\subseteq V(G)$ is a separating set, then $\delta_{G}(X)\cap M = \emptyset$ for any perfect matching $M$ of $G$. 
\end{definition} 
\begin{definition} 
Given $G_1,G_2\in\comp{G}$, 
we say $G_1\paryield{G} G_2$ if there is a separating set $X\subseteq V(G)$ 
such that $V(G_1)\cup V(G_2)\subseteq X$ holds and $G[X]/V(G_1)$ is a factor-critical graph. 
We sometimes denote $\paryield{G}$ simply by $\yield$. 
\end{definition} 
The next theorem is highly analogous to the known  {\em Dulmage-Mendelsohn decomposition} for bipartite graphs~\cite{lovasz2009matching}, in that it describes a partial order over $\comp{G}$:     
\begin{theorem}[Kita~\cite{DBLP:journals/corr/abs-1205-3816, DBLP:conf/isaac/Kita12}]\label{thm:order}
In any factorizable graph, $\yield$ is a partial order over $\comp{G}$. 
\end{theorem} 
Under Theorem~\ref{thm:order}, 
we denote the poset of $\yield$ over $\comp{G}$ by $\poset{G}$. 
For $H\in\comp{G}$, the set of upper bounds of $H$ in $\poset{G}$ is denoted by $\parupstar{G}{H}$. 
The union of vertex sets of all upper bounds of $H$ is denoted by $\vparupstar{G}{H}$. 
We denote $\parupstar{G}{H}\setminus \{H\}$ by $\parup{G}{H}$ 
and $\vparupstar{G}{H}\setminus V(H)$ by $\vparup{G}{H}$. 
We sometimes write them by omitting the subscripts ``$G$''.   
\begin{definition} 
Given $u,v\in V(G)$, 
we say $u\gsim{G} v$ if $u$ and $v$ are contained in the same factor-component and 
$G-u-v$ has no perfect matching. 
\end{definition}  
\begin{theorem}[Kita~\cite{DBLP:journals/corr/abs-1205-3816, DBLP:conf/isaac/Kita12}]\label{thm:sim}
In any factorizable graph $G$, $\gsim{G}$ is an equivalence relation on $V(G)$. 
Each equivalence class is contained in the vertex set of a factor-component. 
\end{theorem} 
Given $H\in\comp{G}$, we denote by $\pargpart{G}{H}$ the family of equivalence classes of $\gsim{G}$ 
that are contained in $V(H)$. Note that $\pargpart{G}{H}$ gives a partition of $V(H)$.  
The structure given by Theorem~\ref{thm:sim} is called the {\em generalized Kotzig-Lov\'asz partition} 
as it is a generalization of the results given by Kotzig~\cite{kotzig1959z1,kotzig1960z3,kotzig1959z2} and  Lov\'asz~\cite{lovasz1972structure}.  

Even though Theorems~\ref{thm:order} and \ref{thm:sim} were established independently, 
a natural relationship between the two is shown by the next theorem.  

\begin{theorem}[Kita~\cite{DBLP:conf/isaac/Kita12,DBLP:journals/corr/abs-1205-3816}]\label{thm:cor}
Let $G$ be a factorizable graph, and let $H\in\comp{G}$. 
Let $K$ be a connected component of $G[\vparup{G}{H}]$. 
Then, there exists $S\in\pargpart{G}{H}$ with $\parNei{G}{K}\cap V(H) \subseteq S$. 
\end{theorem}

Intuitively, Theorem~\ref{thm:cor} states that each proper upper bound of a factor-component $H$ 
is tagged with a single member from $\pargpart{G}{H}$. 
As a result of Theorem~\ref{thm:cor}, 
the two structures given by Theorems~\ref{thm:order} and \ref{thm:sim} are unified naturally 
to produce a new canonical decomposition that enables us to analyze a factorizable graph as a building-like structure in which each factor-component serves as a floor and each equivalence class serves as a foundation. 

As given in Theorem~\ref{thm:cor}, 
for $H\in\comp{G}$ and $S\in\pargpart{G}{S}$,  we define $\parup{G}{S} \subseteq \parup{G}{H}$  as follows: 
$I\in\parup{G}{H}$ is in $\parup{G}{S}$ 
if the connected component $K$ of $G[\vparup{G}{H}]$ with $V(I)\subseteq V(K)$ 
satisfies $\parNei{G}{K}\cap V(H)\subseteq S$.   
The union of vertex sets of factor-components in $\parup{G}{S}$ is denoted by $\vparupstar{G}{S}$. 
The sets $\vparupstar{G}{S}\setminus S$ and $\vparupstar{G}{H}\setminus\vparupstar{G}{S}$ are denoted by $\vparup{G}{S}$ and  $\vparcoup{G}{S}$, respectively. 
Note that the family $\{\vupstar{S} : S\in\pargpart{G}{H}\}$ (resp. $\{ \vup{S} : S\in\pargpart{G}{H} \}$) gives a partition of $\vupstar{H}$ (resp. $\vup{H}$).  
We sometimes omit the subscript ``$G$'' if the meaning is apparent from the context.  

In the remainder of this section, we present some pertinent properties that will be used in later sections. 
  
\begin{lemma}[Kita~\cite{DBLP:conf/cocoa/Kita13,DBLP:journals/corr/abs-1212-5960}]\label{lem:tpath} 
Let $G$ be a factorizable graph and $M$ be a perfect matching of $G$, 
and let $H\in\comp{G}$. 
Let $S\in\pargpart{G}{H}$,  and let $T\in\pargpart{G}{H}$ be such with $S\neq T$. 
\begin{romanenumerate} 
\item \label{item:tpath:up2base}
For any $x\in \vupstar{S}$, there exists $y\in S$ such that 
there is an $M$-balanced path from $x$ to $y$ whose vertices except for $y$ are in $\vup{S}$. 
\item \label{item:tpath:base2base}
For any $x\in S$ and any $y\in T$, 
there is an $M$-saturated path between $x$ and $y$ whose vertices are in $\vupstar{H}\setminus\vup{S}\setminus\vup{T}$. 
\item \label{item:tpath:base2coup}
For any $x\in S$ and any $y\in\vcoup{S}$, 
there is an $M$-balanced path from $x$ to $y$ whose vertices are in $\vupstar{H}\setminus \vup{S}$. 
\item \label{item:tpath:up2up}
For any $x\in \vupstar{S}$ and any $y\in\vupstar{T}$ 
there is an $M$-saturated path between $x$ and $y$ whose vertices are in $\vupstar{H}$. 
\end{romanenumerate} 
\end{lemma}

\begin{lemma}[Kita~\cite{DBLP:journals/corr/abs-1205-3816, DBLP:conf/isaac/Kita12}]\label{lem:ear2comp}
Let $G$ be a factorizable graph, and let $M$ be a perfect matching of $G$.  
If there is an $M$-ear relative to $H_1\in\comp{G}$ 
and traversing $H_2\in\comp{G}$, 
then $H_1\yield H_2$ holds. 
\end{lemma} 

From Lemma~\ref{lem:ear2comp}, the next lemma is easily derived. 
\begin{lemma}[Kita~\cite{DBLP:journals/corr/abs-1205-3816, DBLP:conf/isaac/Kita12}]\label{lem:vear2comp}
Let $G$ be a factorizable graph and $M$ be a perfect matching of $G$. 
Let $x\in V(G)$, and let $H\in\comp{G}$ be such with $x\in V(H)$. 
If there is an $M$-ear $P$ relative to $\{x\}$,  
then the connected components of $P-E(H)$ are $M$-ears relative to $H$. 
Hence, if $I\in\comp{G}$ has common vertices with $P$, 
then $H\yield I$ holds. 
\end{lemma} 

\begin{lemma}[Kita~\cite{DBLP:journals/corr/abs-1205-3816, DBLP:conf/isaac/Kita12}]\label{lem:immediate}
Let $G$ be a factorizable graph, 
and $M$ be a perfect matching of $G$. 
If $G_1\in\comp{G}$ is an immediate lower-bound of $G_2\in\comp{G}$ with respect to $\yield$,  then there is an $M$-ear relative to $G_1$ and traversing $G_2$. 
\end{lemma}

\begin{remark}
The results presented in this section are obtained without using any known results 
via a fundamental graph theoretic discussion on matchings. 
\end{remark}

\section{Structure of Towers} \label{sec:tower}
The remainder of this paper introduces the new results. 
In this section, we further develop the theory of canonical decomposition in  Section~\ref{sec:pre:cathedral} 
to acquire lemmas for Section~\ref{sec:newproof}. 
We  define and explore the notions of {\em towers}, {\em arcs},  and {\em tower-sequences},  
and work towards assuring the existence of arcs and tower-sequences with certain maximality, {\em spanning arcs} and {\em spanning tower-sequences}.  
We aim at obtaining Lemmas~\ref{lem:seq2span} and \ref{lem:min2span};  
they will be the main tools in Section~\ref{sec:newproof:casemixed}.

In this section, unless otherwise stated, let $G$ be a factorizable graph and $M$ be a perfect matching. 
The set of minimal elements in the poset $\poset{G}$ is denoted by $\pmin{G}$. 
\begin{definition}
Let $H\in\comp{G}$. 
A {\em tower } over $H$ is the subgraph  $G[\vupstar{H}]$
and is denoted by $\partower{G}{H}$ or simply by $\tower{H}$. 

Given $H_1, H_2 \in\comp{G}$ such that neither $H_1\yield H_2$ nor $H_2\yield H_1$ hold,  
we say $\tower{H_1}$ and $\tower{H_2}$ are {\em tower-adjacent} 
or {\em t-adjacent} if 
$\vup{H_1}\cap \vup{H_2} \neq \emptyset$ 
or $E[\vupstar{H_1}, \vupstar{H_2}] \neq \emptyset$ hold.  
Here, $S_1\in\pargpart{G}{H_1}$ is a {\em port} of this adjacency 
if $\vup{S_1}\cap \vup{H_2} \neq \emptyset$ 
or $E[\vupstar{S_1}, \vupstar{H_2}] \neq \emptyset$ hold. 
\end{definition}

The next lemma will be used for Lemma~\ref{lem:adjacent2arc} as well as for Lemma~\ref{lem:path2ear}. 
\begin{lemma}\label{lem:towerrelative}
Let $G$ be a factorizable graph, 
and $M$ be a perfect matching of $G$. 
For any $H\in\comp{G}$, 
there is no non-trivial $M$-ear relative to $\tower{H}$. 
Any trivial $M$-ear relative to $V(\tower{H})$ is an edge of $\tower{H}$. 
\end{lemma} 

\begin{proof}
Let $P$ be a non-trivial $M$-ear relative to $\tower{H}$. 
Let $x_1$ and $x_2$ be the ends of $P$, 
and let $S_1, S_2\in\pargpart{G}{H}$ be such with $x_1\in \vupstar{S_1}$ and $x_2\in \vupstar{S_2}$. 
By Lemma~\ref{lem:tpath} \ref{item:tpath:up2base}, 
there is an $M$-balanced path $Q_i$ from $x_i$ to a vertex $t_i\in S_i$ 
with $V(Q_i)\setminus \{t_i\} \subseteq \vup{S_i}$ for each $i\in\{1,2\}$. 
Trace $Q_1$ from $x_1$, and let $z_1$ be the first encountered vertex that is in a factor-component $I$ with $V(I)\cap V(Q_2)\neq \emptyset$. 
Trace $Q_2$ from $x_2$, and let $z_2$ be the first encountered vertex in $I$. 
Note that $Q_i$ is an $M$-balanced path from $x_i$ to $t_i$ for each $i\in\{1,2\}$. 
Then, $x_1Q_1z_1 + P + x_2Q_2z_2$ is an $M$-ear relative to $I$ and traversing the factor-components that $P$ traverses. This implies by Lemma~\ref{lem:ear2comp} that the factor-components traversed by $P$ are upper bounds of $I\in \upstar{H}$ and accordingly, of $H$, too. This is a contradiction. 
This proves the first statement. 
The remaining statement is obvious. 
\end{proof}

\begin{definition}
Let $H_1, H_2\in \comp{G}$ 
be such with $H_1\neq H_2$.  
An $M$-exposed path $P$ is an $M$-{\em arc} between $H_1$ and $H_2$ 
if the ends of $P$ are in $H_1$ and $H_2$ 
whereas the other vertices are disjoint from $H_1$ and $H_2$. 
\end{definition}

The next two lemmas describe properties on $M$-arcs. 
\begin{lemma}\label{lem:adjacent2arc}
Let $G$ be a factorizable graph, and $M$ be a perfect matching of $G$. 
Let $H_1, H_2\in\comp{G}$ be such that neither $H_1\yield H_2$ nor $H_2\yield H_1$ hold. 
If $\tower{H_1}$ and $\tower{H_2}$ are t-adjacent, with ports $S_1\in\pargpart{G}{H_1}$ and $S_2\in\pargpart{G}{H_2}$,   
then there is an $M$-arc between $H_1$ and $H_2$, whose ends are in $S_1$ and $S_2$ 
whereas the other vertices  are contained in $\vup{S_1}\cup \vup{S_2}$.  
\end{lemma} 

\begin{proof}
Let $uv\in E[\vupstar{S_1}, \vupstar{S_2}\setminus \vupstar{S_1}]$, 
where $u\in \vupstar{S_1}$ and $v\in \vupstar{S_2}\setminus \vupstar{S_1}$. By Lemma~\ref{lem:tpath} \ref{item:tpath:up2base}, 
there is an $M$-balanced path $P_2$ from $v$ to a vertex $w\in S_2$ with $V(P_2)\setminus \{w\} \subseteq \vup{S_2}$. 
Additionally, 
there is an $M$-balanced path $P_1$ from $u$ to a vertex $z\in S_2$ with $V(P_1)\setminus \{z\} \subseteq \vup{S_1}$ 
\begin{pclaim}\label{claim:adjacent2arc:disjoint}
The paths $P_1$ and $P_2$ are disjoint. 
\end{pclaim}
\begin{proof} 
Suppose this claim fails. 
First, note that $v\not\in \vupstar{H_1}$; 
otherwise,  $v\in\coup{S_1}$ holds, and this contradicts Theorem~\ref{thm:cor} or the assumption $H_1\neq H_2$.  
Trace $P_2$ from $v$, and let $r$ be the first encountered vertex in $\vupstar{H_1}$. 
Then,  $uv + vP_2r$ is a non-trivial $M$-ear relative to $\vupstar{H_1}$, 
which contradicts Lemma~\ref{lem:towerrelative}. 
\end{proof} 
By Claim~\ref{claim:adjacent2arc:disjoint}, 
$P_1 + uv + P_2$ is a desired $M$-arc. 
\end{proof}

\begin{lemma}\label{lem:arcdisjoint}
Let $G$ be a factorizable graph, and $M$ be a perfect matching of $G$. 
Let $H_1, H_2\in\comp{G}$ be such that neither $H_1\yield H_2$  nor $H_2\yield H_1$ hold. 
Let $P$ be an $M$-arc between $H_1$ and $H_2$, 
whose ends are $u_1\in S_1$ and $u_2\in S_2$, where $S_1 \in \pargpart{G}{H_1}$ and $S_2\in \pargpart{G}{H_2}$. 
Then, $P$ is disjoint from $\vcoup{S_1}\cup\vcoup{S_2}$. 
\end{lemma}
\begin{proof}
Suppose the statement fails, 
and let $x$ be the first vertex in, e.g., $\vcoup{S_1}$ that is encountered when we trace $P$ from $u_1$; $P$ is an $M$-ear relative to $\vcoup{S_1}\cup S_1$. 

By Lemma~\ref{lem:tpath} \ref{item:tpath:base2coup}, 
there is an $M$-saturated path $Q$ between $u_1$ and $x$ with $V(Q)\subseteq \vcoup{S_1}\cup S_1$. 
Then, $P + Q$ is an $M$-saturated path that contains non-allowed edges in $\delta(\vcoup{S_1}\cup S_1)$. By Lemma~\ref{lem:circ2matching}, 
this is a contradiction. 
Hence $P$ is disjoint from $\vcoup{S_1}$, and also, by symmetry, from $\vcoup{S_2}$. 
\end{proof}

\begin{definition}
Let $H_1, \ldots, H_k \in \comp{G}$, where $k\ge 1$.  
For each $i\in \{1\ldots, k\}$, let $S_i^+, S_i^- \in \pargpart{G}{H_i}$ be such with $S_i^+ \neq S_i^-$.  
We say $H_1, \ldots, H_k$ is a {\em tower-sequence}, {\em from $H_1$ to $H_k$},  
if $k=1$ holds  
or if $k>1$ holds and for each $i \in \{ 1, \ldots, k-1 \}$, 
$\tower{H_i}$ and $\tower{H_{i+1}}$ are t-adjacent with ports $S_i^+$ and $S_{i+1}^-$. 
\end{definition} 

The next lemma states that no repetition occurs in a tower-sequence.   
\begin{lemma}\label{lem:minseq2arc}
Let $G$ be a factorizable graph, 
and $M$ be a perfect matching of $G$. 
Let $H_1,\ldots, H_k \in\pmin{G}$, where $k > 1$, be a tower-sequence with ports $S_i^+, S_i^-\in\pargpart{G}{H_i}$ for $i\in\{1,\ldots, k\}$. 
Then, 
\begin{romanenumerate}
\item \label{item:minseq2arc:distinct}
$H_i\neq H_j$ holds for any $i,j\in \{1,\ldots, k\}$ with $i\neq j$, 
and 
\item \label{item:minseq2arc:arc} 
there is an $M$-arc between $H_1$ and $H_k$ 
whose ends are in $S_1^+$ and $S_k^-$ 
and which,  if $k\ge 3$ holds,  traverses each $H_2, \ldots, H_{k-1}$.  
\end{romanenumerate}  
\end{lemma} 

\begin{proof}
We proceed by induction on $k$. 
If $k=2$, 
then \ref{item:minseq2arc:distinct} and \ref{item:minseq2arc:arc} hold by the definition of tower-sequences and by Lemma~\ref{lem:adjacent2arc}. 
Let $k > 2$, and suppose \ref{item:minseq2arc:distinct} and \ref{item:minseq2arc:arc} hold for $1,\ldots, k-1$.  
By applying induction hypothesis to the substructures $H_1,\ldots, H_{k-1}$ and $H_2,\ldots, H_k$,  we obtain that $H_i\neq H_j$ holds for any $i,j\in\{1,\ldots, k\}$ 
with $i\neq j$ and $\{i,j\}\neq\{1,k\}$.  
Consider the subsequence $H_1,\ldots, H_{k-1}$.
There is an $M$-arc $\hat{P}$ between $H_1$ and $H_{k-1}$ that satisfies~\ref{item:minseq2arc:arc}. Let $\hat{s}\in S_1^+$ and $\hat{t}\in S_{k-1}^-$ be the ends of $\hat{P}$. 
By Lemma~\ref{lem:adjacent2arc}, 
there is an $M$-arc $P$ between $H_{k-1}$ and $H_k$, whose ends are $s\in S_{k-1}^+$ and $t\in S_{k}^-$, such that its vertices except for $s$ and $t$ are in $\vup{S_{k-1}^+}\cup\vup{S_k^-}$. 
By Lemma~\ref{lem:tpath} \ref{item:tpath:base2base}, 
there is an $M$-saturated path $Q$ between $\hat{t}$ and $s$ with $V(Q)\subseteq \vupstar{H_{k-1}}\setminus \vup{S_{k-1}^-} \setminus \vup{S_{k-1}^+}$. 
Let $\hat{Q} := P + Q$; 
then, $\hat{Q}$ is an $M$-balanced path from $\hat{t}$ to $t$ that traverses $H_{k-1}$. 
\begin{pclaim}\label{claim:minseq2arc:share1vertex}
The paths $\hat{P}$ and $\hat{Q}$ have only $\hat{t}$ as a common vertex. 
\end{pclaim} 
\begin{proof}
Suppose this claim fails, 
and let $x$ be the first vertex in $\hat{P}$ that is encountered if we trace $\hat{Q}$ from $\hat{t}$. 
If $x\hat{P}\hat{t}$ has an even number of edges, then $x\hat{P}\hat{t} + \hat{t}\hat{Q}x$ is an $M$-alternating circuit that contains non-allowed edges in $\delta(H_{k-1})$. This is a contradiction by Lemma~\ref{lem:circ2matching}. 
Otherwise, if  $x\hat{P}\hat{t}$ has an odd number of edges, then $x\hat{P}\hat{t} + \hat{t}\hat{Q}x$ is an $M$-ear relative to $x$ and traversing $H_{k-1}$. 
This implies by Lemma~\ref{lem:ear2comp} that $H_{k-1}$ has a lower-bound in $\poset{G}$ that is distinct from itself.  This is again a contradiction. 
\end{proof} 
By Claim~\ref{claim:minseq2arc:share1vertex},  
$\hat{P}+\hat{Q}$ is an $M$-exposed path that traverses  $H_2,\ldots, H_{k-1}$. 
If $H_1 = H_k$ holds, then $\hat{P}+\hat{Q}$ is an $M$-ear relative to $H_1$. This is a contradiction by Lemma~\ref{lem:ear2comp}, because $H_2, \ldots, H_{k-1}\in\pmin{G}$.  
Hence, we obtain $H_1\neq H_k$, and so $H_1,\ldots, H_k$ are all mutually distinct. Accordingly, 
$\hat{P}+\hat{Q}$ is an $M$-arc satisfying the statement.  
\end{proof}

\begin{definition}
A factor-component $H\in\pmin{G}$ is a {\em border} of $G$ 
if $\tower{H}$ is t-adjacent with no other tower or if  
exactly one member $S$ from $\pargpart{G}{H}$ 
can be a port by which  $\tower{H}$ is t-adjacent with other towers, 
i.e., $E[\vupstar{S}, V(G)\setminus \vupstar{H}]\neq \emptyset$ 
and $E[\vupstar{T}, V(G)\setminus \vupstar{H}] = \emptyset$ for any $T\in\pargpart{G}{H}\setminus\{S\}$ hold.  
Here, $S$ is the {\em port} of the border $H$. 
We denote the set of borders of $G$ by $\border{G}$. 

\end{definition}

\begin{definition}
We say a tower-sequence $H_1,\ldots, H_k\in\pmin{G}$,  where $k\ge 1$,  is {\em spanning} if $H_1$ and $H_k$ are borders of $G$.  
An $M$-arc between $H\in\comp{G}$ and $I\in\comp{G}$ is {\em spanning} if 
$H$ and $I$ are borders of $G$. 
\end{definition} 

Finally, we can derive the lemmas on spanning tower-sequences and spanning $M$-arcs. 
From Lemma~\ref{lem:minseq2arc}, the next lemma is obtained rather easily. 
\begin{lemma}\label{lem:seq2span} 
Let $G$ be a  factorizable graph. 
For a tower-sequence $H_1, \ldots, H_k\in\pmin{G}$, 
there is a spanning tower-sequence $I_1,\ldots, I_l\in\pmin{G}$ with $l \ge k$ and $I_i = H_1, \ldots, I_{i+k}= H_k$ for some $i\in\{1,\ldots, k-l\}$.  
\end{lemma} 

\begin{proof} 
We first prove the following claim: 
\begin{pclaim}\label{claim:seq2span:add} 
If $H_1, \ldots, H_k\in\pmin{G}$ is not spanning, 
there exists a tower-sequence $H_1,\ldots, H_k, H_{k+1}\in\pmin{G}$ or $H_0, H_1,\ldots, H_k\in\pmin{G}$. 
\end{pclaim} 
\begin{proof} 
Because it is not a spanning tower-sequence, 
either $H_1$ or $H_k$ is a non-border; say, let $H_k\not\in\border{G}$. 
Let $S_i^+, S_i^-\in\pargpart{G}{H_i}$ be the ports of $H_1,\ldots, H_k$ 
for $i\in\{1,\ldots, k\}$. 
There exists $T\in\pargpart{G}{H_i}$ with $T\neq S_k^-$ such that 
$\tower{H_k}$ is t-adjacent with another tower over $I\in\pmin{G}$ with $T$ being a port. 
Then, $H_1,\ldots, H_k, H_{k+1}$, where $H_{k+1} = I$, is a tower-sequence. 
\end{proof} 
According to  Claim~\ref{claim:seq2span:add}, 
given a non-spanning tower-sequence, we can repeat extending it  by  adding an element. 
By Lemma~\ref{lem:minseq2arc} \ref{item:minseq2arc:distinct}, 
this repetition ends at some point, 
and a spanning tower-sequence is obtained. 
\end{proof}

The next lemma follows as an easy consequence of Lemmas~\ref{lem:seq2span} and \ref{lem:minseq2arc}. 
\begin{lemma}\label{lem:min2span} 
Let $G$ be a factorizable graph and $M$ be a perfect matching of $G$,  and let $H\in\pmin{G}$. 
\begin{romanenumerate} 
\item \label{item:min2span:seq} 
There exists a spanning tower-sequence $H_1, \ldots, H_k\in\pmin{G}$ with $H = H_i$ for some $i\in\{1,\ldots, k\}$.
\item \label{item:min2span:arc} 
There exists a spanning $M$-arc that has common vertices with $H$.  
\end{romanenumerate} 
\end{lemma} 

\begin{proof} 
Consider the tower-sequence that consists solely of  $H$. 
By Lemma~\ref{lem:seq2span}, Statement \ref{item:min2span:seq} is obtained. 
Moreover, by Lemma~\ref{lem:minseq2arc} \ref{item:minseq2arc:arc}, 
Statement \ref{item:min2span:arc} is obtained. 
\end{proof}

\begin{remark}
From Lemma~\ref{lem:min2span}, the following is implied: 
if the poset $\poset{G}$ of a factorizable graph $G$ has more than one minimal element, 
then $G$ has at least two distinct borders. 
\end{remark}

The next lemma is about the nature of borders, 
but can be derived without other results in this section. 
This lemma will also be used in Section~\ref{sec:newproof:casemixed}. 
\begin{lemma}\label{lem:deg1path}
Let $G$ be a factorizable graph. 
Let $H\in\border{G}$, 
and let $S$ be the port of $H$.  
Then, the set of vertices that can be reached from $S$ by an $M$-saturated path is $\vcoup{S}$. 
\end{lemma} 

\begin{proof}
According to Lemma~\ref{lem:tpath} \ref{item:tpath:base2coup}, 
there is an $M$-saturated path between each vertex in $S$ and each vertex in $\vcoup{S}$. Hence, it suffices to show that there is no $M$-saturated path between any vertex in $S$ and any vertex in $V(G)\setminus\vcoup{S}$. 
Suppose that for vertices $x\in S$ and  $z\in V(G)\setminus \vcoup{S}$, there is an $M$-saturated path $Q$ between $x$ and $z$. By Lemma~\ref{lem:delete2path}, $z\not\in S$ holds. 
Trace $Q$ from $x$. Obviously, the vertex that we encounter immediately after $x$ is in $V(H)\setminus S$. 
Keep tracing $Q$, and let $w$ be the first encountered vertex in $V(G)\setminus \vcoup{S}\setminus S$,  
and let $r$ be the vertex immediately before $w$; $r$ is in $S$, and the edge $rw$ is not in $M$.  
Hence, the path $xQr$ is an $M$-saturated path between $x\in S$ and $r\in S$, 
which is a contradiction by Lemma~\ref{lem:delete2path}.   
\end{proof}

\section{A New Proof of the Tight Cut Lemma}\label{sec:newproof} 
\subsection{General Statements} \label{sec:newproof:gen}
In  this Section~\ref{sec:newproof}, we introduce our new proof of the Tight Cut Lemma. 
Here, in Section~\ref{sec:newproof:gen}, we present definitions and assumptions that will be used throughout the remainder of the paper. We also explain the organization of the new proof and provide some lemmas.

\begin{rtight}
Let $\hat{G}$ be a brick, 
and $\hat{S}\subseteq  V(\hat{G})$ be such with $1 < |\hat{S}| < |V(\hat{G})|-1$.  
Then, there is a perfect matching with more than one edge in $\delta_{\hat{G}}(\hat{S})$. 
\end{rtight} 
Let $\hat{G}$ and $\hat{S}$ be as given  above. 
We need to prove that 
$\delta_{\hat{G}}(\hat{S})$ is not a tight cut. 
Let $\hat{M}$ be a perfect matching of $\hat{G}$. 
If $|\delta_{\hat{G}}(\hat{S})\cap \hat{M}| > 1$ holds, then we have nothing to do.  
Hence, in the following,  we  assume $|\delta_{\hat{G}}(\hat{S})\cap \hat{M}| = 1$ 
and prove $\hat{S}$ is not a tight cut by finding a {\em $\hat{S}$-fat} perfect matching, i.e., 
a perfect matching with more than one edge in $\delta_{\hat{G}}(\hat{S})$.    

Let  $\complement{\hat{S}}$ be $V(\hat{G})\setminus \hat{S}$.  
Let $u\in \hat{S}$ and $v\in \complement{\hat{S}}$ be such that $ \delta_{\hat{G}}(\hat{S}) \cap \hat{M}  = \{uv\}$. 
We denote $\hat{G}-u-v$ by $G$, $\hat{S}-u$ by $S$, $\complement{\hat{S}} -v$ by $\complement{S}$, 
and $\hat{M} -uv$ by $M$.   

Note that $G$ is connected and has a perfect matching $M$. 
Additionally, $\delta_{G}(S)\cap M = \emptyset$ holds in $G$.   
If $S$ is not a separating set, then of course $\delta_{\hat{G}}(\hat{S})$ is not a tight cut in $G$, 
and we are done. Therefore, in the following, we assume that 
\begin{quote} 
$S$ is a separating set of $G$ 
\end{quote}  
and prove the Tight Cut Lemma for this case. 

Without loss of generality, we also assume in the following that 
\begin{quote} 
$G$ has a border whose vertices are contained in $S$. 
\end{quote} 
According to Lemma~\ref{lem:circ2matching}, 
if we  find an $\hat{M}$-alternating circuit $C$ of $\hat{G}$ 
with more than one edges in $\delta_{\hat{G}}(\hat{S}) \setminus \hat{M}$, 
then a $\hat{S}$-fat perfect matching is obtained by taking $E(C)\triangle \hat{M}$. 
We find such an $\hat{M}$-alternating circuit  
by analyzing the matching structure of $G$ using the canonical decomposition in Section~\ref{sec:pre:cathedral}. 
The succeeding Sections~\ref{sec:newproof:casemixed} and \ref{sec:newproof:casewhole} 
correspond to proofs of the respective case analyses:  
\begin{itemize} 
\item In Section~\ref{sec:newproof:casemixed}, a proof is given for  the case where $\pmin{G}$ also has a factor-component whose vertex set is contained in $\complement{S}$;  
\item Section~\ref{sec:newproof:casewhole} is the counterpart to Section~\ref{sec:newproof:casemixed} and gives a proof for the case where every factor-component in $\pmin{G}$ has the vertex set contained in $S$, which completes the new proof of the Tight Cut Lemma. 
\end{itemize}

In the following, we present lemmas that will be used by Sections~\ref{sec:newproof:casemixed} and \ref{sec:newproof:casewhole} when we find a cut vertex in $G$. 
\begin{lemma}\label{lem:cut2nei}
Let $x$ be a cut vertex of $G$, and let $C$ be one of the connected components of $G-x$. 
Then, $\parNei{\hat{G}}{w}\cap V(C)\neq\emptyset$ holds for each $w\in \{u,v\}$. 
\end{lemma} 

\begin{proof}
Suppose that the claim fails, i.e., 
suppose $\parNei{\hat{G}}{w}\cap V(C) = \emptyset$, where $w\in\{u,v\}$. 
It follows that $\{z, x\}$, where $z\in\{u,v\}\setminus\{w\}$, is a vertex-cut of $\hat{G}$, 
which leaves $C$ as one of the connected components of $\hat{G}-\{z, x\}$. 
This is a contradiction, because $\hat{G}$ is 3-connected.  
\end{proof} 

\begin{lemma}\label{lem:cut2saturated}
Let $x$ be a cut vertex of $G$, and let $C$ be one of the connected components of $G-x$. 
If $V(C)\cup \{x\}$ is a separating set of $G$, 
then,  for each $w\in \{u, v\}$, 
there exists $y\in V(C)\cap\parNei{\hat{G}}{w}$ such that $G$ has an $M$-saturated path between $x$ and $y$. 
\end{lemma} 

\begin{proof}
Let $z\in \{u, v\}\setminus\{w\}$. 
As $\hat{G}$ is a brick, there is an $\hat{M}$-saturated path, $P$, between $x$ and $z$. 
If we trace $P$ from $z$, then the second vertex on $P$ is $w$ and the third vertex, $y$, 
is such with $y\in \parNei{\hat{G}}{w}\cap V(C)$ and that $xPy$ is an $\hat{M}$-saturated path, 
for which $V(xPy)\subseteq V(C)$ holds. Therefore, $xPy$ gives a desired path. 
\end{proof}

\subsection{When there exists a factor-component in $\pmin{G}$ whose vertices are in $\complement{S}$}
\label{sec:newproof:casemixed} 
In  this Section~\ref{sec:newproof:casemixed}, 
we assume that $\pmin{G}$ has a factor-component, which may be or may not be a border,  whose vertex set is contained in $\complement{S}$.  
We prove the Tight Cut Lemma for this case,  using mainly the results obtained in Section~\ref{sec:tower}.  
The next lemma is obtained from Lemmas~\ref{lem:deg1path} and \ref{lem:cut2nei}  and will be used in the proof of Lemma~\ref{lem:arc2matching}.  
\begin{lemma}\label{lem:deg1nei}
Let $H\in\border{G}$, and let $S\in\pargpart{G}{H}$ be the port of $H$. 
Then, 
$\vparcoup{G}{S}\cap \parNei{\hat{G}}{w} \neq \emptyset$ holds for each $w\in\{u,v\}$. 
\end{lemma} 

\begin{proof}
First, consider the case where $S$ is a singleton, which consists of $x\in V(H)$. 
Then, $x$ is a cut vertex of $G$, 
and for some connected components $C_1,\ldots, C_k$ of $G-x$ (in fact $k=1$ holds),  
$V(C_1)\cup\cdots\cup V(C_k)\cup\{x\}$ is equal to $\vparcoup{G}{S}$. 
Therefore, Lemma~\ref{lem:cut2nei} proves the statement for this case. 

Next, consider the case with $|S| >  1$. 
Let $x, y\in S$ be such with $x\neq y$. 
As $\hat{G}$ is a brick, it has a $\hat{M}$-saturated path $P$ between $x$ and $y$. 
By Lemma~\ref{lem:delete2path}, $P$ is not a path of $G$, which implies $uv\in E(P)$. 
Let $z_1\in\parNei{\hat{G}}{u}\cap V(P) \setminus \{v\}$ and $z_2\in \parNei{\hat{G}}{v}\cap V(P) \setminus \{u\}$, and assume, without loss of generality, that  $x$, $z_1$, $z_2$, $y$ appear in this order if we trace $P$ from $x$. 
Then,  $xPz_1$ and $yPz_2$ are $M$-saturated paths of $G$.  

From Lemma~\ref{lem:deg1path}, 
we obtain $z_1, z_2\in \vparcoup{G}{S}$. 
The lemma is proven. 
\end{proof} 

Lemma~\ref{lem:deg1nei} derives the next lemma, 
which provides the main strategy to find a desired $\hat{M}$-alternating circuit.  

\begin{lemma}\label{lem:arc2matching}
If $G$ has a spanning $M$-arc with an edge in $E_G[S,\complement{S}]$, 
then $\hat{G}$ has a $\hat{S}$-fat perfect matching.
\end{lemma}  

\begin{proof}
Let $P$ be a spanning $M$-arc with $E(P)\cap E_G[S,\complement{S}]\neq\emptyset$, 
between two borders $H_1$ and $H_2$.   
Let $s_1$ and $s_2$ be the ends of $P$, and 
let $S_1\in\pargpart{G}{H_1}$ and $S_2\in\pargpart{G}{H_2}$ be such with $s_1\in S_1$ and $s_2\in S_2$. 
Without loss of generality, we can assume $V(H_1)\subseteq S$; 
if $V(H_1)\subseteq \complement{S}$ holds, 
then we can exchange the roles of $S$ and $\complement{S}$ 
without contradicting  the assumption on $\border{G}$. 
According to Lemma~\ref{lem:deg1nei}, 
there exist $t_1\in \parNei{\hat{G}}{v} \cap \vparcoup{G}{S_1}$ 
and $t_2\in\parNei{\hat{G}}{u}\cap \vparcoup{G}{S_2}$. 
By Lemma~\ref{lem:tpath} \ref{item:tpath:base2coup},  
there exists an $M$-saturated path $Q_i$ from $t_i$ to $s_i$  with $V(Q_i)\setminus \{s_i\}\subseteq \vparcoup{G}{S_i}$ for each $i\in\{1,2\}$. 
According to Lemma~\ref{lem:arcdisjoint}, 
$P$ is disjoint from $Q_1$ and $Q_2$ except for the ends.  
Therefore,  $C$ is an $\hat{M}$-alternating circuit of $\hat{G}$, 
where $C:= Q_1 + t_1v + uv + vt_2 + Q_2 + P$. 

If $t_1\in S$ holds, then $t_1v \in E_{\hat{G}}[\hat{S}, \complement{\hat{S}}]\setminus \hat{M}$ holds 
and therefore $|E(C)\cap E_{\hat{G}}[\hat{S}, \complement{\hat{S}}\setminus \hat{M}]|\ge 2$ follows. 
Hence, from Lemma~\ref{lem:circ2matching}, $\hat{M}\triangle E(C)$ is a $\hat{S}$-fat perfect matching. 
Otherwise, if $t_1\in \complement{S}$ holds, 
then $Q_1$ has an edge in $E_G[S, \complement{S}]$ because the other end $s_1$ is in $S$. Hence, again,  $|E(C)\cap E_{\hat{G}}[\hat{S}, \complement{\hat{S}}]\setminus \hat{M}|\ge 2$ follows, 
and the statement is proven for this case, too. 
This completes the proof of this lemma. 
\end{proof}

As Lemma~\ref{lem:arc2matching} is obtained, 
we give the following two lemmas to find such a spanning $M$-arc.

\begin{lemma}\label{lem:halfopen} 
If $G$ also has a border whose vertices are in $\complement{S}$, 
then there is a spanning $M$-arc that has an edge in $E_G[S,\complement{S}]$. 
\end{lemma} 

\begin{proof}
Define $\mathcal{H}_1\subseteq \pmin{G}$ (resp. $\mathcal{H}_2\subseteq \pmin{G}$)  as follows: 
$H\in\pmin{G}$ is in $\mathcal{H}_1$ (resp. $\mathcal{H}_2$) if 
there is a tower-sequence from a border whose vertex set is contained in $S$ (resp. $\complement{S}$) 
to $H$. 

By Lemma~\ref{lem:min2span} \ref{item:min2span:seq},  
$\mathcal{H}_1\cup\mathcal{H}_2 =\pmin{G}$. 
\begin{pclaim}\label{claim:halfopen:intersect} 
The two sets  $\mathcal{H}_1$ and $\mathcal{H}_2$ intersect. 
\end{pclaim} 
\begin{proof} 
Suppose that this claim fails, namely, that $\mathcal{H}_1\cap\mathcal{H}_2 = \emptyset$ holds. 
As $G$ is connected, the two sets of vertices $\bigcup_{H\in\mathcal{H}_1} \vparupstar{G}{H}$ and $\bigcup_{H\in\mathcal{H}_2} \vparupstar{G}{H}$ either intersect, or are disjoint with some edges joining them. 
This implies that there exist $H_1\in\mathcal{H}_1$ and $H_2\in \mathcal{H}_2$ such that $\partower{G}{H_1}$ and $\partower{G}{H_2}$ are t-adjacent. 
By Lemma~\ref{lem:min2span} \ref{item:min2span:seq}, 
there is a spanning tower-sequence $I_1,\ldots, I_k\in\pmin{G}$ with $k\ge 2$ and $H_1 = I_i$ for some $i\in\{1,\ldots,k\}$. 
If $V(I_1)\subseteq \complement{S}$ or $V(I_k)\subseteq \complement{S}$ hold, 
then $H_1\in \mathcal{H}_1\cap\mathcal{H}_2$, which is a contradiction. Otherwise, if $V(I_1)\subseteq S$ and $V(I_k)\subseteq S$, then  
either $I_1,\ldots,I_i, H_2$ or $H_2, I_i,\ldots, I_k$ is a tower-sequence. 
Thus, $H_2 \in \mathcal{H}_1\cap\mathcal{H}_2$ holds, which is again a contradiction. \end{proof}

\begin{pclaim}\label{claim:halfopen:spanseq}
There is a spanning tower-sequence from a border whose vertex set is contained in $S$ 
to a border whose vertex set is contained in $\complement{S}$.  
\end{pclaim} 
\begin{proof} 
By Claim~\ref{claim:halfopen:intersect}, there exists $H\in \mathcal{H}_1 \cap \mathcal{H}_2$. 
By $H\in\mathcal{H}_1$, there is a tower-sequence from $H_1\in\border{G}$  to $H$ with $V(H_1)\subseteq S$. 
Hence, by Lemma~\ref{lem:seq2span}, there is a spanning tower-sequence $H_1,\ldots, H_k\in\pmin{G}$ 
with $k\ge 2$ and $H = H_i$ for some $i\in\{1,\ldots,k\}$. 
If $V(H_k)\subseteq \complement{S}$ holds, we are done; 
thus, let $V(H_k)\subseteq S$. 
By $H\in\mathcal{H}_2$, there is a tower-sequence $I_1,\ldots, I_l\in\pmin{G}$ with $l\ge 1$, $I_1\in\border{G}$, $V(I_1)\subseteq \complement{S}$, and $I_l = H$. 
Either $H_1,\ldots, H_i = H = I_l, \ldots, I_1$ or $H_k,\ldots, H_i = H = I_l, \ldots, I_1$ forms a spanning tower-sequence, 
satisfying the statement of this claim. 
\end{proof} 

By Lemma~\ref{lem:minseq2arc} \ref{item:minseq2arc:arc} and Claim~\ref{claim:halfopen:spanseq}, 
we obtain a desired spanning $M$-arc. 
\end{proof}

The next lemma treats the counterpart case to Lemma~\ref{lem:halfopen}. 

\begin{lemma}\label{lem:surrounded} 
Assume every border of $G$ has the vertex set that is contained in $S$. 
If there exists a non-border element of $\pmin{G}$ whose vertex set is contained in $\complement{S}$, then there is a spanning $M$-arc that has some edges in $E_G[S,\complement{S}]$. 
\end{lemma}   

\begin{proof}
Let $H\in\pmin{G}$ be such with $V(H)\subseteq \complement{S}$. 
As given in Lemma~\ref{lem:min2span} \ref{item:min2span:arc},  
take a spanning $M$-arc $P$ with $V(P)\cap V(H)\neq \emptyset$. 
The ends of $P$ are in $S$, so $P$ has at least two edges in $E_G[S,\complement{S}]$. 
\end{proof}

Regardless of whether there is a border that has the vertex set in $\complement{S}$ or not, 
Lemmas~\ref{lem:halfopen} and \ref{lem:surrounded} assure that 
$G$ has an $M$-arc with an edge in $\delta_{G}(S)$. 
Hence, from Lemma~\ref{lem:arc2matching}, 
we conclude that $\hat{G}$ has a $\hat{S}$-perfect matching, 
and the Tight Cut Lemma is proven for the case of Section~\ref{sec:newproof:casemixed}.

\subsection{When every factor-component in $\pmin{G}$ has the vertex set contained in $S$}
\label{sec:newproof:casewhole}
\subsubsection{Shared Assumptions and Lemmas}\label{sec:newproof:casewhole:shared}
Here in Section~\ref{sec:newproof:casewhole}, 
we assume that the vertex set of any factor-component in $\pmin{G}$ is contained in $S$ 
and prove the Tight Cut Lemma under this assumption.    

Section~\ref{sec:newproof:casewhole:shared} explains the assumptions, definitions, and lemmas that will be used throughout Section~\ref{sec:newproof:casewhole}. 
Let $S_0\subseteq S$ be the inclusion-wise maximal separating subset of $S$ such that $\{H_1, \ldots, H_p\}$ is a lower-ideal of $\poset{G-S}$, where $S_0 = V(H_1)\dot{\cup}\cdots\dot{\cup} V(H_p)$. 
Arbitrarily choose  a connected component $C$ of $G-S_0$. 
Sections~\ref{sec:newproof:casewhole:cut} and \ref{sec:newproof:casewhole:proper}  give the proofs of the Tight Cut Lemma for the cases where $|\parNei{G}{C}\cap S_0| = 1$ holds and does not hold, respectively.  
Note that $V(C)$ is a separating set in $G$ 
and  $C$ is factorizable.  
In addition, note that for each $H\in\pmin{C}$, $V(H)\subseteq \complement{S}$ holds  
by the definition of $S_0$.   
The following two lemmas will be used in both of the succeeding case analyses. 
\begin{lemma}\label{lem:min2ear}\label{claim:min2ear}
For each $H\in\pmin{C}$, $G$ has an $M$-ear, $P_H$, relative to $S_0$ and traversing $H$. 
\end{lemma}
\begin{proof} 
Because $H\not\in\pmin{G}$ holds here, 
$\poset{G}$ has an immediate lower-bound element $I\in\comp{G}$ of $H$. 
By Lemma~\ref{lem:immediate}, there is an $M$-ear $P$ relative to $I$ and traversing $H$.  
\begin{pclaim}\label{claim:min2ear:lb}
The vertex set of $I$ is contained in $S_0$. 
\end{pclaim} 
\begin{proof} 
Suppose this claim fails, i.e., suppose $V(I)\subseteq V(G)\setminus S_0$. 
If there exists $H'\in \comp{G}$ with $V(H')\subseteq S_0$ such that $P$ traverses $H'$, 
then $I\paryield{G} H'$ holds by Lemma~\ref{lem:ear2comp}, which contradicts the definition of $S_0$. 
Hence, $V(I)\cup V(P)\cup V(H)\subseteq V(G)\setminus S_0$ holds, 
and accordingly,  $V(I)\cup V(P)\cup V(H)\subseteq V(C)$ holds. 
This implies $I\paryield{C} H$, which contradicts $H\in\pmin{C}$. 
Hence, $V(I)\subseteq S_0$ follows. 
\end{proof} 
Under Claim~\ref{claim:min2ear:lb}, 
 the connected components of $P - E(G[S_0])$ are $M$-ears relative to $S_0$, and one of them, $P_H$, traverses $H$. 
\end{proof} 

Under Lemma~\ref{lem:min2ear}, 
for each $H\in\pmin{C}$, 
arbitrarily choose and fix  an $M$-ear relative to $S_0$ and traversing $H$;  
in the remainder of this paper, we denote it by $P_H$. 
\begin{lemma}\label{lem:path2ear}\label{claim:path2ear} 
Let $y\in V(C)$,  and let $H\in\pmin{C}$ be such that $y\in\vparupstar{C}{H}$.  
Then, there is an $M$-balanced path $Q_H^y$  from $y$ to $x_H$, one of the ends of the $M$-ear $P_H$, with $V(Q_H^y)\setminus \{x_H\} \subseteq V(C)$.
\end{lemma} 

\begin{proof}
Consider the  possibly identical connected components $R_1$ and $R_2$ of $P_H-E(C[\vparupstar{C}{H}])$ that contain the ends of $P_H$; if $P_H$ is proper, let us denote it by $R_1$, and let $R_2$ be an empty graph.   
Let $z_1, z_2\in \vparupstar{C}{H}$ be the ends of $R_1$ and $R_2$ that are distinct from the ends of $P_H$. 
Let $R' := P_H - E(R_1 + R_2)$.  
We have $E(R')\subseteq E(C[\vparupstar{C}{H}])$; 
otherwise, the connected components of $R'-E(C[\vparupstar{C}{H}])$  are non-trivial $M$-ears relative to $\vparupstar{C}{H}$ 
or are trivial $M$-ears that contradict Lemma~\ref{lem:towerrelative}.  
Accordingly, 
$T_1\neq T_2$ holds; 
otherwise, the connected components of $R'-\vparup{C}{T_1}$ are $M$-saturated paths whose ends are all in $T_1$, which contradicts Lemma~\ref{lem:delete2path}. 

Let $T_3\in\pargpart{C}{H}$ be such with $y\in \vparupstar{C}{T_3}$. 
Either $T_1$ or $T_2$ is not identical to $T_3$; without loss of generality, let $T_1\neq T_3$. 
Let $x_H$ be the end of $P_H$ that is in $V(R_1)$.  
By Lemma~\ref{lem:tpath} \ref{item:tpath:up2up}, there is an $M$-saturated path $L$ between $y$ and $z_1$ 
with $V(L)\subseteq \vparupstar{C}{H}$.   
The path $L + z_1R_1x_H$ is a desired path $Q_H^y$.  
\end{proof} 

Following the above, 
in the remainder of this paper, for each $H\in\pmin{C}$ and each $y\in\vparupstar{C}{H}$, 
let $Q_H^y$ be the path as given in Lemma~\ref{lem:path2ear}, 
and let $x_H$ be the end of the $M$-ear $P_H$ that is also an end of the path $Q_H^y$. 

\subsubsection{Case with $|\parNei{G}{C}\cap S_0| = 1$} \label{sec:newproof:casewhole:cut}
Here in Section~\ref{sec:newproof:casewhole:cut}, 
we assume that there exists $x_0\in S_0$ with $\parNei{G}{C}\cap S_0 = \{x_0\}$, and 
prove the Tight Cut Lemma under this assumption. 
In this case, of course $P_H$ is a non-proper $M$-ear with the unique end $x_H$, 
which is accordingly equal to $x_0$ for each $H\in\pmin{C}$. 
\begin{lemma}\label{lem:cut2matching} 
If $|\parNei{G}{C}\cap S_0 | = 1 $ holds,  
then $\hat{G}$ has a $\hat{S}$-fat perfect matching. 
\end{lemma} 

\begin{proof}
Note that  $x_0$ is a cut vertex of $G$ such that 
$C$ is a connected component of $G-x_0$. 
Therefore, by Lemma~\ref{lem:cut2saturated}, 
there exists $z\in\parNei{\hat{G}}{v}$ such that there is an $M$-saturated path $R$ between $x_0$ and $z$ with $V(R)\subseteq S_0$. 
Note $zv\in E_{\hat{G}}[\hat{S}, \complement{\hat{S}}]\setminus \hat{M}$.

By Lemma~\ref{lem:cut2nei}, there exists $y\in V(C)\cap \parNei{\hat{G}}{u}$. 
Let $H\in\pmin{C}$ be such with $y\in \vparupstar{C}{H}$, 
and take a path $Q_H^y$ as given in Lemma~\ref{lem:path2ear}. 
Let $K := R + zv + uv + vy + Q_H^y$. Note that $K$ is an $\hat{M}$-alternating circuit of $\hat{G}$.

If $y\in \complement{S}$ holds, then $yu\in E_{\hat{G}}[\hat{S}, \complement{\hat{S}}]\setminus \hat{M}$ holds. 
Otherwise, if $y\in S$ holds, 
then $Q_H^y$ has edges in $E_G[S, \complement{S}]$ and, accordingly, in $E_{\hat{G}}[\hat{S}, \complement{\hat{S}}]\setminus \hat{M}$; 
this is because Lemma~\ref{lem:path2ear} assures that $Q_H^y$ traverses $\complement{S}$ 
while the ends of $Q_H^y$ are in $S$. 
Therefore, in each case, $K$ has at least two edges in $E_{\hat{G}}[\hat{S}, \complement{\hat{S}}]\setminus \hat{M}$. 
Hence, by Lemma~\ref{lem:circ2matching}, 
 $\hat{M}\triangle E(K)$ is a $\hat{S}$-fat perfect matching. 

\end{proof} 

From Lemma~\ref{lem:cut2matching}, 
the proof of the Tight Cut Lemma for the case analysis of Section~\ref{sec:newproof:casewhole:cut} is completed.  
\subsubsection{Case with $|\parNei{G}{C}\cap S_0| > 1$}\label{sec:newproof:casewhole:proper} 
Here, in Section~\ref{sec:newproof:casewhole:proper}, 
we treat the counterpart case to Section~\ref{sec:newproof:casewhole:cut}; 
namely, we assume $|\parNei{G}{C}\cap S_0| > 1$. 
We use the next lemma as the main strategy to obtain a desired perfect matching: 
\begin{lemma}\label{lem:proper2matching}
If $G$ has a proper $M$-ear relative to $S_0$ and traversing $\complement{S}$, 
then $\hat{G}$ has an $\hat{S}$-fat perfect matching. 
\end{lemma} 

\begin{proof}
Let $P$ be a proper $M$-ear relative to $S_0$ and traversing $\complement{S}$, 
and let $x$ and $y$ be the ends of $P$, with $x\neq y$. 
As $\hat{G}$ is a brick, Lemma~\ref{lem:delete2path} implies that 
it has an  $\hat{M}$-saturated path $Q$ between $x$ and $y$.  
This $Q$ is not a path in $G$, otherwise $P+Q$ is an $M$-alternating circuit of $G$ containing non-allowed edges in $\delta_G(S_0)$, which is a contradiction by Lemma~\ref{lem:circ2matching}. Hence, $uv\in E(Q)$ holds. 
Let $Q_1$ and $Q_2$ be the connected components of $Q-u-v$; 
note that they are $M$-saturated paths. 
\begin{pclaim}\label{claim:proper2matching:disjoint} 
The paths $Q_1$ and $Q_2$ are disjoint from $P$, except for the ends $x$ and $y$. 
\end{pclaim}
\begin{proof} 
Without loss of generality, let $x$ be one of the ends of $Q_1$. 
Suppose the claim fails, 
and let $z$ be the first encountered vertex in $P-x$ if we trace $Q_1$ from $x$.

If $xPz$ has an even number of edges, then $xQ_1z + xPz$ is an $M$-alternating circuit of $G$ containing non-allowed edges in $\delta_G(S_0)$. This contradicts Lemma~\ref{lem:circ2matching}. 
Otherwise, if  $xPz$ has an odd number of edges, then $xQ_1z + xPz$ is an $M$-ear relative to $z$. 
From Lemma~\ref{lem:vear2comp}, this implies that there exist $H_z\in\comp{G}$ with $V(H_z)\subseteq V(G)\setminus S_0$ and $H_0\in\comp{G}$ with $V(H_0)\subseteq S_0$ such that $H_z\paryield{G} H_0$ holds.   
This contradicts the definition of $S_0$. 
Hence, the statement is obtained for $Q_1$.  
With the symmetrical argument, the statement also holds for $Q_2$ 
\end{proof}
By Claim~\ref{claim:proper2matching:disjoint}, 
$P+Q$ forms an $M$-alternating circuit, 
and it has at least two edges in $E_G[S,\complement{S}]$, 
because $P$ has. 
Hence, $\hat{M}\triangle E(P+Q)$ is a desired $\hat{S}$-perfect matching.  
\end{proof} 

As given Lemma~\ref{lem:proper2matching}, we aim at finding such a proper $M$-ear. 
If the $M$-ear $P_H$  is proper for some $H\in\pmin{C}$, 
then Lemma~\ref{lem:proper2matching} gives a $\hat{S}$-fat matching of $\hat{G}$. 
Hence, in the remainder of this proof, we assume that 
\begin{quote}
$P_H$ is not proper, having the unique end $x_H$, 
for each $H\in\pmin{C}$.  
\end{quote}
The next two lemmas find desired $M$-ears and therefore $\hat{S}$-fat perfect matchings 
under the assumptions that are the counterparts to each other. 
\begin{lemma}\label{lem:multiNei}\label{claim:multiNei}
Let $H\in\pmin{C}$. 
If $\parNei{G}{\vparupstar{C}{H}}\cap S_0$ contains a vertex other than $x_H$,  
then $\hat{G}$ has a $\hat{S}$-fat matching. 
\end{lemma} 

\begin{proof}
Let $z$ be a vertex in $\parNei{G}{\vparupstar{C}{H}}\cap S_0$ that is distinct from $x_H$, 
and let $y \in \vparupstar{C}{H}$ be such with $zy \in E(G)$. 
Take a path $Q_H^y$ as in Lemma~\ref{lem:path2ear}.  
Then $Q_H^y + zy$ is an $M$-ear relative to $S_0$, with the two distinct  vertices $x_H$ and $y$, 
and traversing $V(H)\subseteq \complement{S}$. 
Hence, by Lemma~\ref{lem:proper2matching}, this lemma is now proven. 
\end{proof}  

As the counterpart of Lemma~\ref{lem:cut2matching}, 
the next lemma treats the case where $\parNei{G}{\vparupstar{C}{H}}\cap S_0 = \{x_H\}$ for any $H\in\pmin{C}$. 
Note that according to the assumption of Section~\ref{sec:newproof:casewhole:proper}, 
there exist $H, I\in\pmin{C}$ with $x_H\neq x_I$. 
\begin{lemma}\label{lem:multi2ear} 
If $\parNei{G}{\vparupstar{C}{H}}\cap S_0 = \{x_H\}$ holds for any $H\in\pmin{C}$, 
then $\hat{G}$ has a $\hat{S}$-fat perfect matching. 
\end{lemma} 

\begin{proof}
As $C$ is connected, 
there exist $H_1, H_2\in\pmin{C}$ with $x_{H_1}\neq x_{H_2}$ such that $\partower{C}{H_1}$ and $\partower{C}{H_2}$ are t-adjacent. 
Let $S_i\in\pargpart{C}{H_i}$ be the ports of this adjacency for $i\in\{1,2\}$. 
From Lemma~\ref{lem:adjacent2arc}, 
we obtain an $M$-arc $R$ 
whose vertices except for the ends are in $\vparup{C}{S_1}\cup\vparup{C}{S_2}$. 
Let $s_1\in S_1$ and $s_2\in S_2$ be the ends of $R$. 
Take an $M$-balanced path $Q_{H_i}^{s_i}$ from $s_i$ to $x_{H_i}$ as stated in Lemma~\ref{lem:path2ear} for each $i\in\{1,2\}$. 
\begin{pclaim}\label{claim:arc2path}  
The path $Q_{H_i}^{s_i}$ is disjoint from $R$ 
for each $i\in\{1,2\}$.  
\end{pclaim} 
\begin{proof} 
Suppose this claim fails. 
Trace $Q_{H_i}^{s_i}$ from $s_i$, 
and let $t$ be the first encountered vertex in $R$. 
If $s_iRt$ has an even number of edges, 
then $s_iQ_{H_i}^{s_i}t + tRs_i$ is an $M$-alternating circuit of $G$ containing non-allowed edges in $\delta_{C}(S_i)$. 
 This contradicts Lemma~\ref{lem:circ2matching}. 
If  $s_iRt$ has an odd number of edges,  
then $s_iQ_{H_i}^{s_i}t + tRs_i$ is an $M$-ear relative to $t$ and traversing $H_i$. 
Under Lemma~\ref{lem:vear2comp}, this implies that $H_1$ is not minimal in $\poset{C}$, 
which is a contradiction. 
The claim is now proven. 
\end{proof} 

\begin{pclaim}\label{claim:path2path}
The paths $Q_{H_1}^{s_1}$ and $Q_{H_2}^{s_2}$ are disjoint. 
\end{pclaim} 
\begin{proof}
Suppose the claim fails, 
and let $t$ be the first encountered vertex in $Q_{H_2}^{s_2}$ if we trace $Q_{H_1}^{s_1}$ from $s_1$. 
By Claim~\ref{claim:arc2path}, 
$K := R + s_1Q_{H_1}^{s_1}t +  tQ_{H_2}^{s_2}s_2$ is a circuit.

If $tQ_{H_2}^{s_2}s_2$ has an even number of edges, then $K$ is an $M$-ear relative to $t$ and traversing $H_1$. 
By Lemma~\ref{lem:vear2comp}, 
this contradicts $H_1\in\pmin{C}$. 
Otherwise, if $tQ_{H_2}^{s_2}s_2$ has an odd number of edges, then 
$K$ is an $M$-alternating circuit containing non-allowed edges in $\delta_{C}(H_1)$. By Lemma~\ref{lem:circ2matching}, this is a contradiction. 
\end{proof} 
From Claims~\ref{claim:arc2path} and \ref{claim:path2path}, 
$R + Q_{H_1}^{s_1} + Q_{H_2}^{s_2}$ forms an $M$-ear relative to $S_0$, possessing the distinct ends $x_{H_1}$ and $x_{H_2}$, and traversing $H_1$ and $H_2$, which are contained in $\complement{S}$.  
Therefore, the proof is now completed by Lemma~\ref{lem:proper2matching}. 
\end{proof} 

This completes the proof for the case of Section~\ref{sec:newproof:casewhole:proper}.  
Therefore, the whole proof of the Tight Cut Lemma is completed.

\bibliographystyle{splncs03.bst}
\bibliography{library.bib} 

\end{document}